\documentclass[11pt]{amsart}
\usepackage{color}
\usepackage{amsfonts}
\usepackage{amssymb}
\usepackage{amsmath}
\usepackage{amsthm}
\usepackage{amsopn}
\usepackage{graphicx}
\usepackage{times}
\usepackage{url}
\DeclareMathOperator{\Sym}{Sym}

\DeclareMathOperator{\vertt}{vert}
\DeclareMathOperator{\CUT}{CUT}

\newcommand{\lin}{\operatorname{lin}}

\DeclareMathOperator{\vol}{vol}
\DeclareMathOperator{\Lift}{Lift}
\DeclareMathOperator{\DV}{V}
\DeclareMathOperator{\Orth}{O}
\DeclareMathOperator{\conv}{conv}
\DeclareMathOperator{\Iso}{Iso}
\DeclareMathOperator{\treg}{Reg}
\DeclareMathOperator{\tface}{Face}


\newcommand{\Z}{\mathbb Z}
\newcommand{\Q}{\mathbb Q}
\newcommand{\R}{\mathbb R}

\newcommand{\NP}{\mathrm{NP}}

\newcommand{\RR}{\ensuremath{\mathbb{R}}}

\newcommand{\ZZ}{\ensuremath{\mathbb{Z}}}

\newcommand{\sharpp}{$\#\operatorname{P}$}

\def\QuotS#1#2{\leavevmode\kern-.0em\raise.2ex\hbox{$#1$}\kern-.1em/\kern-.1em\lower.25ex\hbox{$#2$}}

\newtheorem{theorem}{Theorem}
\newtheorem{problem}{Problem}
\newtheorem{proposition}{Proposition}

\newtheorem{lemma}{Lemma}

\begin{document}

\title{Complexity and algorithms for computing Voronoi cells of lattices}

\date{\today}

\author{Mathieu Dutour Sikiri\'c}
\address{M. Dutour Sikiri\'c, Rudjer Boskovi\'c Institute, Bijenicka 54, 10000 Zagreb, Croatia}
\email{mdsikir@irb.hr}

\author{Achill Sch\"urmann}
\address{A. Sch\"urmann, Mathematics Department, University of Magdeburg, 
         39106 Magdeburg, Germany}
\email{achill@math.uni-magdeburg.de}

\author{Frank Vallentin}
\address{F. Vallentin, Centrum voor Wiskunde en Informatica (CWI), Kruislaan 413, 1098 SJ Amsterdam, The Netherlands}
\email{f.vallentin@cwi.nl} 

\thanks{The work of the first author has been supported by the
Croatian Ministry of Science, Education and Sport under contract
098-0982705-2707.
The second and the third author were supported by the Deutsche
Forschungsgemeinschaft (DFG) under grant SCHU 1503/4-2.  The third
author was also supported by the Netherlands Organization for
Scientific Research under grant NWO 639.032.203.
All three authors thank the Hausdorff Research Institute for Mathematics
for its hospitality and support.}

\begin{abstract}
  In this paper we are concerned with finding the vertices of the Voronoi cell
  of a Euclidean lattice.
  Given a basis of a lattice, we prove that computing the number of vertices
  is a \sharpp-hard problem.
  On the other hand we describe an algorithm for this
  problem which is especially suited for low dimensional (say
  dimensions at most $12$) and for highly-symmetric lattices.  We use our
  implementation, which drastically outperforms those of current computer
  algebra systems, to find the
  vertices of Voronoi cells and quantizer constants of some prominent
  lattices.
\end{abstract}


\maketitle

\section{Introduction}
\label{sec:introduction}

Let $L = B\Z^m \subseteq \R^n$ be a {\em lattice} of rank $m$ in Euclidean
space given by a matrix $B \in \R^{n \times m}$ of rank $m$.
By $\lin L$ we denote the linear subspace spanned by the elements of $L$.
The {\em Voronoi cell} of $L$ is
\begin{equation*}
\DV(L) = \{x \in \lin L : \text{$\|x\| \leq \|x - v\|$ for all $v \in
L$}\}.
\end{equation*}
The Voronoi cell of a lattice is a centrally symmetric, convex polytope.
The polytopes $V(L)+v$ for $v\in L$ tile $\lin L$.
The study of Voronoi cells is motivated by the fact
that most important geometric lattice parameters have a direct 
interpretation in terms of the Voronoi cell: The {\em determinant} $\det L$
equals the volume of $\DV(L)$, the {\em packing radius} $\lambda(L)$ equals
the inradius of $\DV(L)$, the {\em covering radius} $\mu(L)$ equals the
circumradius of $\DV(L)$, and the {\em quantizer constant} $G(L)$ is
\begin{equation*}
G(L) = (\det L)^{-(1+2/n)} \int_{\DV(L)} \|x\|^2 dx.
\end{equation*}

In this paper we consider theoretical and practical aspects of the
computation of the covering radius as well as
the quantizer constant of a lattice.
These two parameters have many applications, we just name a few:
By computing the covering radius, we get an upper bound for the lattice
sphere covering problem, which is the problem of minimizing the covering
radius among the
lattices of fixed determinant (see \cite[Chapter 2]{conwaysloane}
and \cite{sv-2006}).
The computation of the covering radius of the Leech lattice in
\cite[Chapter 23]{conwaysloane} had a major impact on the study of
hyperbolic reflection groups (see \cite[Chapter 27]{conwaysloane}).
An upper bound for the Frobenius number of a set of integers 
can be obtained from the covering radius of a suitable lattice
(see \cite{frobenius}).
A recent application comes from public key cryptography; Micciancio
\cite{micciancio-2004} found a new connection between the
average-case complexity of finding the packing radius and the
worst-case complexity of determining the covering radius.
In information theory, the quality of a lattice as a vector
quantizer is determined by its quantizer constant
(see \cite{gg-1992,almost,sloanequantiz} and
\cite[Chapter 2.3, Chapter 21]{conwaysloane}).

The structure of this paper is as follows.
In Section~\ref{sec:complexity} we discuss the computational complexity
of the covering radius problem. We prove that the related problem of
counting vertices of the Voronoi cell is \sharpp-hard.  As a byproduct
of our construction, we show that the lattice isomorphism problem
is at least as difficult as the graph isomorphism problem.
We turn to practical algorithms for the covering radius problem in
Section~\ref{sec:algorithm}. There we describe an algorithm which
computes the vertices of the Voronoi cell of a lattice.
Based on this algorithm we give an algorithm for computing the
quantizer constant in Section~\ref{sec:quantizer}.
In Section~\ref{sec:results} we report on computations with our
implementation.
We determine the exact covering radius and
quantizer constants of many prominent lattices which were not
known before.

\section{Computational complexity}\label{sec:complexity}

We formulate the covering radius problem as a decision problem.

\begin{problem} Covering radius problem
\begin{description}
\item[Input] $m$, $n$, $B \in \Q^{m \times n}$, $\mu \in \Q$.
\item[Output] \textsf{Yes}, if $\mu(B\Z^n) \leq \mu$, \textsf{No} otherwise.
\end{description}
\end{problem}

It is conjectured (see \cite[Section 1.1]{micciancio-2004}) that the
covering radius problem is $\NP$-hard. Haviv and Regev
\cite{haviv-regev-2006} showed that there is a constant $c_p$ so that
the covering radius in the $l_p$-norm is $\Pi_2$-hard to approximate
within a constant less than $c_p$ for any large enough $p$. In
\cite{gmr-2005} Guruswami, Micciancio and Regev proved that
approximating it within a factor of $O(\sqrt{m/\log m})$ for a lattice
of rank $m$ cannot be $\NP$-hard unless the polynomial hierarchy
collapses.

Currently, there is only one known general and practical method to
compute $\mu(L)$ for a lattice $L$: First one enumerates the vertices
of $\DV(L)$ and then one finds the vertex with largest norm.
The number of vertices of $\DV(L)$ can be as
large as $(m+1)!$ and furthermore, as we show in Theorem~\ref{sharpp},
even computing this number is \sharpp-hard.

\begin{problem}\label{prob:complex_vertices}
Vertices of a lattice Voronoi cell
\begin{description}
\item[Input] $m$, $n$, $B \in \Q^{m \times n}$.
\item[Output] Number of vertices of $\DV(B\Z^n)$.
\end{description}
\end{problem}

\begin{theorem}
\label{sharpp}
The problem ``Vertices of a lattice Voronoi cell'' is \sharpp-hard.
\end{theorem}

It will be obvious from the proof that we could restrict the problem
to the case $m = n$.
We reduce the problem ``Acyclic orientations of a graph'',
which Linial \cite{linial-1986} showed to be \sharpp-complete,
to Problem \ref{prob:complex_vertices}.

\begin{problem}
\label{prob:acyclic}
Acyclic orientations of a graph
\begin{description}
\item[Input] A graph $G = (V, E)$.
\item[Output] The number of orientations of $G$ with no directed circuit. 
\end{description}
\end{problem}

The structure of the proof of Theorem~\ref{sharpp} is as follows: In
Section~\ref{graph lattices} we construct a matrix $B$ with columns
indexed by $E$ defining a lattice $L(G) = B\Z^E$ from $G$ in
polynomial time. Then we show that the vertices of the
Voronoi cell of $V(L(G))$ are in bijection with the acyclic
orientations of $G$. To establish this bijection we need several
intermediate steps. In Section~\ref{graphs hyperplanes} we associate
to $G$ a hyperplane arrangement $H(G)$ whose chambers
are in bijection with the acyclic orientations of $G$.
In Section~\ref{hyperplanes zonotopes} we recall that the
chambers of a hyperplane arrangement are in bijection with
the vertices of a zonotope associated to the hyperplane arrangement.
These two steps are standard and we cover them rather briefly.
In Section~\ref{lattices zonotopes} we show that the Voronoi
cell of $L(G)$ is a zonotope which, up to a linear transformation, is
the one associated to the hyperplane arrangement $H(G)$.
In Section \ref{byproduct1}, as a byproduct of this construction,
we show that the lattice isomorphism problem is at least as difficult
as the graph isomorphism problem.
Some related complexity results concerning vertex enumeration of
polyhedra given by linear inequalities are in \cite{khachian,dyer83}.

\subsection{From graphs to lattices}
\label{graph lattices}

Let $G=(V,E)$ be a connected graph with vertex set $V=\{1,\ldots,n\}$ and
edge set $E$.
We consider the following orientation of the edges of $G$:
The head of an edge $e = \{v,w\} \in E$ is $e^+ = \max\{v,w\}$
and the tail is $e^- = \min\{v,w\}$.

Let $T \subseteq E$ be the edge set of a spanning tree of $G$, and let
$e \in T$. Deleting $e$ from $T$ divides $T$ into two connected
components with vertex sets $T^+_e$ and $T^-_e$, where $e^+ \in T^+_e$
and $e^- \in T^-_e$. Define the vector $b_{T,e} \in \Z^E$ by
\begin{equation*}
b_{T,e}(f) = 
\left\{
\begin{array}{cl}
1, & \text{if $f^+ \in T^+_e$ and $f^- \in T^-_e$,}\\
-1, & \text{if $f^- \in T^+_e$ and $f^+ \in T^-_e$,}\\
0, & \text{otherwise.}
\end{array}
\right.
\end{equation*}
Then
\begin{equation*}
L(G,T) = \left\{\sum_{e \in T} \alpha_e b_{T,e} : \alpha_e \in \Z \right\} \subseteq \Z^E
\end{equation*}
is a lattice of rank $n-1$.

\begin{proposition}
Let $T$ and $T'$ be spanning trees of $G$.
Then, $L(G,T)=L(G,T')$.
\end{proposition}

\begin{proof}
Since one can connect any two spanning trees by a sequence of
transformations of the form $T \leftrightarrow T \setminus \{e\} \cup
\{f\}$ it suffices to prove the proposition
for $T' = T \setminus \{e\} \cup \{f\}$.
Let $g \in T'$.
If $g=f$, then $b_{T',f} = \pm b_{T,e}$.
If $g \in T$, then denote by $C$ the cycle containing $e$ and $f$.
If $g\notin C$ then $b_{T',g}=b_{T,g}$.
The subgraph of $G$ with edge set $T\setminus \{e,g\}$ has three connected
components, denoted by $C_1$, $C_2$, $C_3$.
Given $h=\{v,w\}\in E$, the value of $b_{T', g}(h)$, $b_{T, g}(h)$
and $b_{T, e}(h)$ depends only on which connected component $v$ and $w$
belong to.
So, in computing $b_{T', g}$, we can reduce ourselves to the
case when $G$ is the complete graph on $\{1,2,3\}$, $g=\{1, 3\}$,
$e=\{1, 2\}$ and $f=\{2, 3\}$.
Easy computation gives $b_{T', g}=b_{T, g}+b_{T,e}$ and
so we conclude that $b_{T',g} = b_{T,g} + \alpha b_{T,e}$ with
$\alpha \in \{-1,0,+1\}$.
\end{proof}

In the following we omit the spanning tree $T$ from the notation
$L(G,T)$ and just write $L(G)$. Note that one can find a basis of $L(G)$
given $G$ in polynomial time.

\subsection{From graphs to hyperplane arrangements}
\label{graphs hyperplanes}

A matrix
\begin{equation}\label{matrix V}
V = (v_1, \ldots, v_m) \in \R^{n \times m}
\end{equation}
with non-zero column vectors $v_i \in \R^n$ gives an {\em arrangement of hyperplanes}
\begin{equation*}
H(V) = \{H_1, \ldots, H_m\}\quad\mbox{with}\quad  H_i = \{c \in \R^n : c \cdot v_i = 0\}. 
\end{equation*}
The hyperplane arrangement $H(V)$ divides the space $\R^n$ into
polyhedral cones, called {\em regions}, of different dimensions.
The regions are partially ordered by inclusion and full-dimensional regions
are called {\em chambers}.

To associate a hyperplane arrangement $H(G)$ with $G$ we consider the
incidence matrix $D_G \in \R^{V \times E}$ of $G$ which is given by
\begin{equation*}
D_G(v,e) =
\left\{
\begin{array}{cl}
1, & \text{if $v = e^+$,}\\
-1, & \text{if $v = e^-$,}\\
0, & \text{otherwise.}
\end{array}
\right.
\end{equation*}
Then we define the hyperplane arrangement of $G$ by $H(G) = H(D_G)$.

In \cite[Lemma 7.1]{greene-zaslavsky-1983} Greene and Zaslavsky show
that the chambers of $H(G)$ are in bijection with the
acyclic orientations of $G$: Let $\Vec{E}$ be an acyclic orientation of
$E$.
Then a chamber of $H(G)$ is given by
\begin{equation*}
\treg(\Vec{E}) = \{x \in \R^V : \text{$x_v < x_w$ if $(v,w) \in \Vec{E}$}\}.
\end{equation*}
Let $R$ be a chamber of $H(G)$. Then an acyclic
orientation of $E$ is given by
\begin{equation*}
\Vec{E}(R) = \{(v,w) : \text{$\{v,w\} \in E$ and $x_v < x_w$ for every $x \in
R$}\}.
\end{equation*}
Obviously, $\treg(\Vec{E}(R)) = R$.

\subsection{Hyperplane arrangements and zonotopes}
\label{hyperplanes zonotopes}

The matrix $V$ in \eqref{matrix V} defines a {\em zonotope} $Z(V)$ by
\begin{equation*}
Z(V) = \left\{\sum_{i=1}^m \alpha_i v_i : -1 \leq \alpha_i \leq 1\right\}.
\end{equation*}
The faces of $Z(V)$ are partially ordered by inclusion.
It is a well-known fact (see e.g. \cite[Theorem 7.16]{ziegler-1995}) that the
partially ordered set of regions of the hyperplane arrangement
$H(V)$ is anti-isomorphic to the partially ordered set of 
faces of $Z(V)$:
Let $R$ be a region of $H(V)$. Let $x\in R$.
Then the corresponding face $\tface(R)$ of $Z(V)$ given by
\begin{equation*}
\tface(R) = \left\{y \in Z(V) : x \cdot y = \max_{z \in Z(V)} x \cdot z\right\},
\end{equation*}
does not depend on the choice of $x$.
Let $F$ be a face of $Z(V)$. Let $y$ be in the relative interior of $F$.
Then the corresponding region $\treg(F)$ of $H(V)$ given by
\begin{equation*}
\treg(F)=\left\{x \in \R^n : \max_{z \in Z(V)} x \cdot z = x \cdot y\right\},
\end{equation*}
does not depend on the choice of $y$.
Obviously, $\tface(\treg(F)) = F$ and $F' \subseteq F$ if and
only if $\treg(F') \supseteq \treg(F)$.
In particular, the chambers of $H(V)$ are in bijection
with the vertices of $Z(V)$.

\subsection{From lattices to zonotopes}
\label{lattices zonotopes}

Let $L \subseteq \R^n$ be a lattice. The {\em support} of a vector $v \in L$
is $\underline{v} = \{i \in \{1,\ldots,n\} : v_i \neq 0\}$. The vector $v$ is
called {\em elementary} if $v \in \{-1,0,+1\}^n \setminus \{0\}$ and if $v$
has minimal support among all vectors in~$L \setminus \{0\}$. We
say that two vectors $v, w \in L$ are {\em conformal} if $v_i w_i \geq 0$
for all $i = 1, \ldots, n$. The lattice $L$ is called {\em regular} if
for every vector $v \in L \setminus \{0\}$ there exists an elementary
vector $u \in L$ with $\underline{u} \subseteq \underline{v}$.

\begin{lemma}\label{LemmaRegularLattices}
{\rm (\cite[Chapter 1]{tutte-1971})}

(i) For any graph $G$ the lattice $L(G)$ is regular.

(ii) If $L$ is a regular lattice, then every $v\in L$ can be written as a sum of pairwise conformal elementary vectors.

(iii) If $L$ is a regular lattice, $v\in L$ is elementary, and $u\in L$ satisfies $\underline{u}=\underline{v}$, then there exists a
factor $\alpha \in \Z$ such that $u = \alpha v$.
\end{lemma}

A vector $v \in L$ for which $\DV(L) \cap \{x \in \R^n : x \cdot v =
\frac{1}{2} v \cdot v\}$ is a facet of $\DV(L)$ is called
{\em relevant}.
Voronoi characterizes in \cite[page 277]{voronoi-1908} the relevant
vectors of $L$: A nonzero vector $v \in L$ is relevant if and only if
$\pm v$ are the only shortest vectors in $v + 2L$.

\begin{proposition}
\label{strict Voronoi}
In a regular lattice, a vector is elementary if and only if it is relevant.
\end{proposition}

\begin{proof}
Let $v \in L$ be a relevant vector.
By Lemma \ref{LemmaRegularLattices} (ii), we can write $v = \sum_{k=1}^m w_k$ as a
sum of pairwise conformal elementary vectors $w_k \in L$.  Assume that
$m \geq 2$. Defining $u = v - 2w_1$ gives $u \neq \pm v$ and $u \cdot
u = v \cdot v - 4(v - w_1) \cdot w_1$.  Since the vectors $w_k$,
$k = 1, \ldots, m$, are pairwise conformal we have
$(v - w_1) \cdot w_1 \geq 0$, and $\pm v$ is not the unique shortest
vector in $v + 2L$.
In this case $v$ cannot be a relevant vector. Hence, $m = 1$ and $v$ is
an elementary vector.

Let $v \in L$ be an elementary vector, and let $u \in v + 2L$ be a
lattice vector with $u \neq \pm v$.
We have $v - u \in 2L \subseteq 2\Z^n$ and $v_i \in \{-1,0,+1\}$,
which shows $\underline{v} \subseteq \underline{u}$.
The case $\underline{v} \neq \underline{u}$
immediately leads to $v\cdot v < u\cdot u$.  If $\underline{v} =
\underline{u}$, then by Lemma \ref{LemmaRegularLattices} (iii), 
there exists a factor $\alpha \in \Z \setminus
\{-1,+1\}$ so that $u = \alpha v$, hence $v\cdot v < u\cdot u$.
In both cases $\pm v$ are the only shortest vectors in $v + 2L$.
Hence, $v$ is a relevant vector.
\end{proof}

The following special case of the Farkas lemma is proved e.g.\ in
\cite[Theorem 22.6]{rockafellar-1970}.

\begin{lemma}
\label{farkas}
Let $L \subseteq \R^n$ be a regular lattice.  Let $x \in \R^n$ be a
vector, and let $\alpha_1, \ldots, \alpha_n \in \R \cup
\{\pm\infty\}$.  Either there exists a vector $y' \in (\lin
L)^{\perp}$ lying in $x + \prod_{i=1}^n [-\alpha_i, \alpha_i]$, or
there exists a vector $y \in \lin L$ such that for all $z \in x +
\prod_{i=1}^n [-\alpha_i, \alpha_i]$ the inequality $y \cdot z > 0$
holds.  If the second condition holds, then one can choose $y$ to be
an elementary vector of $L$.
\end{lemma}

\begin{theorem}
\label{projection}
Let $L \subseteq \R^n$ be a regular lattice.  Let $P \in \R^{n \times
n}$ be the matrix of the orthogonal projection of $\R^n$ onto $\lin L$.
Then, $\DV(L) = \frac{1}{2}Z(P) = P([-1/2,1/2]^n)$.
\end{theorem}

\begin{proof}
Suppose that $x \in [-1/2,1/2]^n$. For all $v \in \Z^n \setminus \{0\}$
the inequality $x \cdot v \leq \frac{1}{2} v \cdot v$ holds.
Write $x = y + y'$ with $y = Px \in \lin L$ and
$y' \in (\lin L)^{\perp}$. For all $v \in L \setminus
\{0\}$ we have $y \cdot v = x \cdot v - y' \cdot v \leq
\frac{1}{2} v \cdot v$. Thus, $Px \in \DV(L)$.

Suppose now that $y \in \DV(L)$. If there exists $x \in (-y + [-1/2,
1/2]^n) \cap (\lin L)^{\perp}$, then $y + x \in [-1/2,1/2]^n$ and $P(y
+ x) = y$. Assume that such a vector does not exist. Then by
Lemma~\ref{farkas} there is an elementary lattice vector $v \in L$ so
that $v \cdot (-y + [-1/2,1/2]^n) > 0$. This implies $v \cdot (-y -
\frac{1}{2} v) > 0$. Hence, $-y \not\in \DV(L)$. Since $\DV(L)$ is centrally symmetric,
this contradicts the assumption $y \in \DV(L)$.
\end{proof}

In \cite[Proposition 8.1]{biggs-1997} Biggs shows that for the lattice
$L(G)$ the matrix $P$ can be written in the form $P = XD_G$ where $D_G
\in \R^{V \times E}$ is the incidence matrix of $G$ and
$X \in \R^{E \times V}$ is given by
\begin{equation}
\label{eq:Xev}
X(e,v) = \frac{\text{number of spanning trees $T$ with $e \in T$ and
$v \in T^+_e$}}{\text{number of spanning trees of $G$}}.
\end{equation}
Furthermore, the linear map given by $X$ restricted to the image of
$D_G$ is a bijection.

Thus, the zonotope $Z(P)$ which is the Voronoi cell of $L(G)$ equals
$\frac{1}{2}X Z(D_G)$. Hence, there is a linear isomorphism between
the faces of $\DV(L(G))$ and those of $Z(D_G)$. This completes the
proof of Theorem~\ref{sharpp}.

\medskip

Using a straightforward computation we get the following proposition.

\begin{proposition}\label{PropositionCoveringRadius}
Using the notation in \eqref{eq:Xev}, the covering radius of the lattice
$L(G)$ is given by
\begin{equation}
\label{eq:noidea}
\mu(L(G))^2 = \max_{x \in [-1/2,1/2]^E} \sum_{e \in E} \left(\sum_{f \in E} (X(e,f^+) -
X(e,f^-)) x(f) \right)^2.
\end{equation}
\end{proposition}

Unfortunately, we do not have a combinatorial interpretation of
\eqref{eq:noidea}. Finding one could lead to a proof of the
$\NP$-hardness of the covering radius problem.

\subsection{Lattice isomorphism problem}
\label{byproduct1}

Using the construction $L(G)$ used in the proof of Theorem~\ref{sharpp}, we
reduce the graph isomorphism problem to the lattice isomorphism problem
in polynomial time.
We don't know whether one can give a
reverse polynomial time reduction. For the graph isomorphism problem
no polynomial time algorithm is known. It is generally believed to lie
in $\NP \cap \mbox{co-}\NP$. So it is unlikely that it is
$\NP$-hard. For more information on the computational complexity of
this problem, see the book \cite{kst-1993} of K\"obler, Sch\"oning and
T\'oran.

\begin{problem} Lattice isomorphism problem
\begin{description}
\item[Input] $m,n, B, B' \in \Q^{m \times n}$ matrices of rank $m$.
\item[Output] \textsf{Yes}, if there is an orthogonal transformation
$O$ so that $OB\Z^n = B'\Z^n$, \textsf{No} otherwise.
\end{description}
\end{problem}

\begin{problem} Graph isomorphism problem
\begin{description}
\item[Input] Graphs $G = (V, E_G)$, $H = (V, E_H)$.
\item[Output] \textsf{Yes}, if there is a permutation $\sigma : V \to
V$ so that for all $v, w \in V$ we have $\{v,w\} \in E_G$ if and only
if $\{\sigma(v), \sigma(w)\} \in E_H$, \textsf{No} otherwise.
\end{description}
\end{problem}

\begin{theorem}
There is a polynomial time reduction of the graph isomorphism problem
to the lattice isomorphism problem.
\end{theorem}

\begin{proof}
Let $G = (V, E_G)$ and $H = (V, E_H)$ be graphs. We modify $G$ and $H$
by adding three extra vertices to $V$ each adjacent to all vertices in
$V$. We call the new graphs $G'$ and $H'$ which are by construction
$3$-connected and they are isomorphic if and only if $G$ and
$H$ are isomorphic.

It is clear that the lattices $L(G')$ and $L(H')$ defined in
Subsection \ref{graph lattices} are isomorphic
whenever $G'$ and $H'$ are. For this direction it would be enough to
consider the original graphs $G$ and $H$.

Now suppose that the lattices $L(G')$ and $L(H')$ are isomorphic. We
apply the 2-Isomorphism-Theorem of Whitney (actually we only use the easy
subcase of $3$-connected graphs \cite[Lemma 5.3.2]{oxley-1992}):
Because the graphs $G'$ and $H'$ are $3$-connected and there is a
bijection between the elementary vectors preserving conformality, the
graphs $G'$ and $H'$ are isomorphic.
\end{proof}

\section{Algorithms}
\label{sec:algorithm}

In this section we describe an algorithm which computes all vertices
of a lattice Voronoi cell. 
Our focus is on implementability and practical performance, using the
symmetries of the lattice.
In fact, the algorithm computes all full-dimensional
Delone cells and the adjacencies between them up to equivalence.
We give necessary definitions in
Section~\ref{notation}. In Section~\ref{mainalgorithm} we describe the
algorithm's main steps and in the following sections we give details
about its subalgorithms.
In Section \ref{grammatrices}, we explain how to use Gram matrices
instead of lattice
basis and in Section \ref{sec:comparison} we compare our method
with existing algorithms.

\subsection{Notation}
\label{notation}

From now on, we assume lattices $L \subseteq \R^n$ to have full rank $n$.

To encode the vertices of $\DV(L)$ we use Delone cells.
A point $x \in \R^n$ defines a {\em Delone cell} $D(x)$ by
\begin{equation*}
D(x) = \conv\left\{v \in L : \|x - v\| = \min_{w\in L}\|x - w\|\right\}.
\end{equation*}
Denote by $S(x,r)$ the sphere with center $x$ and
radius $r$.
For $r=\min_{v\in L}\|x - v\|$, the
sphere $S(x,r)$ is called {\em empty}, since there is no lattice
point inside.
In this case the polytope $D(x)$ is the convex hull of $S(x,r)\cap L$.
The Delone cell of
a vertex of $\DV(L)$ is characterized among all Delone cells by the following
properties: The origin is a vertex of $D(x)$ and $D(x)$ is full-dimensional.

It is well known (see e.g. \cite{edelsbrunner}) that the Delone
cells are the projections of the faces of the infinite
$(n+1)$-dimensional polyhedral set
\begin{equation*}
\Lift(L)=\conv\left\{ (x, \Vert x\Vert^2) : x\in L\right\}.
\end{equation*}

The task of finding a vertex of a Delone cell of a point $x$, given a
lattice basis of $L$, is called the {\em closest vector problem}.
Generally this is
an NP-hard problem \cite{RefCVP}; however, there are algorithms
and implementations available which can solve this problem 
rather fast in low dimensions.

The {\em orthogonal group} $\Orth(L)$ of $L$ is the group
of all orthogonal transformations $A \in \Orth(\R^n)$ fixing $L$,
i.e.\ $A(L) = L$.
The {\em isometry group} $\Iso(L)$ of $L$ is the group
generated by $\Orth(L)$ and all lattice translations
$t_v : \R^n \to \R^n$ with $t_v(x) = x + v$ for $v \in L$.

We say that two vertices $x$ and $x'$ of $\DV(L)$ are equivalent if there is
an $A \in \Orth(L)$ so that $A(x) = x'$. Correspondingly, we say that two
Delone cells $D(x)$ and $D(x')$ are equivalent if there is an $A \in \Iso(L)$
so that $A(D(x)) = D(x')$.

\subsection{Main algorithm}
\label{mainalgorithm}

Our algorithm finds a complete list of inequivalent
full-dimensional Delone cells of $L$ with respect to $\Iso(L)$.
The enumeration process is a graph traversal algorithm of the graph
of equivalence classes of full-dimensional Delone cells of $L$.
Two equivalence classes are adjacent whenever there is a facet 
between two of its representatives.
Note that this graph can have loops and multiple edges.

For the graph traversal algorithm below one needs four subalgorithms, 
which we explain in the following sections.

\begin{flushleft}
\smallskip
\textbf{Input:} $n$, $B \in \Q^{n \times n}$ matrix of rank $n$.\\
\textbf{Output:} Set~${\mathcal M}$ of all inequivalent full-dimensional Delone cells of the lattice $B\Z^n$ with respect to the group $\Iso(B\ZZ^n)$.\\
\smallskip
$x \leftarrow $ an initial vertex of $\DV(B\Z^n)$. \hfill (Section~\ref{Initial})\\
$T\leftarrow \{D(x)\}$.\\
${\mathcal M} \leftarrow \emptyset$.\\
\textbf{while} there is a $D \in T$ \textbf{do}\\
\hspace{2ex} ${\mathcal M} \leftarrow {\mathcal M} \cup \{D\}$.\\
\hspace{2ex} $T \leftarrow T \setminus \{D\}$.\\
\hspace{2ex} ${\mathcal F} \leftarrow \mbox{facets of $D$}$.\hfill (Section~\ref{Adjacent})\\
\hspace{2ex} \textbf{for} $F \in {\mathcal F}$ \textbf{do}\\
\hspace{2ex} \hspace{2ex} $D' \leftarrow$ full-dimensional Delone cell with $F=D\cap D'$.\hfill (Section~\ref{Adjacent})\\
\hspace{2ex} \hspace{2ex} \textbf{if} $D'$ is not equivalent to a Delone cell in ${\mathcal M}\cup T$ \textbf{then}\hfill (Section~\ref{Equivalence})\\
\hspace{2ex} \hspace{2ex}\hspace{2ex} $T \leftarrow T \cup \{D'\}$.\\
\hspace{2ex} \hspace{2ex} \textbf{end if}\\
\hspace{2ex} \textbf{end for}\\
\textbf{end while}\\
\end{flushleft}

Two full-dimensional Delone cells $D(x)$ and $v+D(x)$,
both containing the origin, are equivalent under $\Orth(L)$
if and only if $0$ and $-v$ are equivalent under the stabilizer
group of $D(x)$ in $\Iso(L)$.
As a consequence, we can compute the vertices of $\DV(L)$ under $\Orth(L)$
in the following way:
For every orbit of full-dimensional Delone cells given by a 
representative $D(x)$, we compute the orbits of vertices of $D(x)$
under the stabilizer group and get the corresponding
orbits of vertices of $\DV(L)$ under $\Orth(L)$.

\subsection{Finding an initial vertex}
\label{Initial}

Now we explain a method for computing an initial vertex of the Voronoi cell of
a lattice, i.e.\ a full-dimensional Delone cell containing the origin.
The method we propose is a so-called cutting-plane algorithm,
which is a well-known technique in combinatorial optimization.

Let us describe the geometric idea. We start with an outer
approximation of the Voronoi cell given by linear inequalities.
The first outer approximation is the polytope defined by the
inequalities $\pm b_i\cdot x\leq \tfrac{1}{2} b_i\cdot b_i$ for given lattice
basis vectors $b_1, \ldots, b_n$.
Then we find a vertex $x$ of the approximation by linear
programming (see e.g.\ \cite{schrijver-1986}).
Deciding whether the vertex $x$ belongs to the Voronoi cell $\DV(L)$
can be done as follows: Compute the vertices of the Delone cell $D(x)$.
If the origin is a
vertex of $D(x)$, then $x$ is a vertex of $\DV(L)$. Otherwise $x$ is not
contained in $\DV(L)$, and for all vertices $v$ of $D(x)$ we have the strict
inequality $\|x - v\| < \|v\|$.
So the new linear inequalities $v \cdot x \leq \tfrac{1}{2} v \cdot v$
together with the old ones provide a tighter outer
approximation of the Voronoi cell.
Since we started with a compact outer approximation, finitely
many iterations of these steps suffice to find a vertex of the Voronoi cell.

One advantage of this method is that the computation of all facets of the
Voronoi cell is not required, i.e.\ we do not use Voronoi's characterization
(see Section \ref{lattices zonotopes}) of facet defining vectors,
which involves solving exponentially many closest vector problems.
Figure \ref{IterativeVertex} illustrates this algorithm.

\begin{figure}
\begin{center}
\resizebox{10.0cm}{!}{\includegraphics{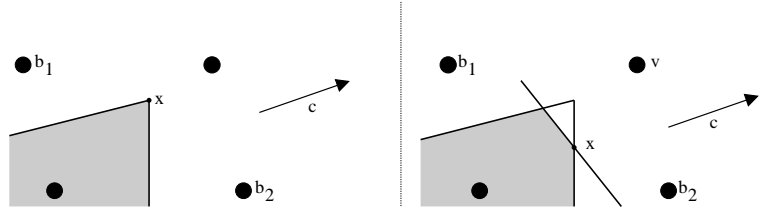}}\par
\end{center}
\caption{Finding an initial vertex of $\DV(L)$}
\label{IterativeVertex}
\end{figure}

\begin{flushleft}
\smallskip
\textbf{Input:} $n$, $B = (b_1, \ldots, b_n) \in \Q^{n \times n}$ matrix of rank $n$.\\
\textbf{Output:} vertex $x$ of $\DV(B\Z^n)$.\\
\smallskip
$c \leftarrow $ random vector in $\Q^n$.\\
${\mathcal B} \leftarrow \{\pm b_1, \ldots, \pm b_n\}$.\\
\textbf{do}\\
\hspace{2ex} $x\leftarrow$ a vertex of the polytope $\{x\mbox{~:~}b \cdot x \leq \frac{1}{2} b \cdot b\mbox{~for~all~} b \in {\mathcal B}\}$,\\
\hspace{2ex} \hspace{4ex} which maximizes $c\cdot x$.\\
\hspace{2ex} ${\mathcal E} \leftarrow $ set of closest lattice vectors in $B\Z^n$ to $x$.\\
\hspace{2ex} \textbf{if} $0 \in {\mathcal E}$ \textbf{then}\\
\hspace{2ex} \hspace{2ex} return $x$.\\
\hspace{2ex} \textbf{end if}\\
\hspace{2ex} ${\mathcal B} \leftarrow {\mathcal B} \cup {\mathcal E}$.\\
\textbf{end do}\\
\end{flushleft}

\subsection{Computing facets of, and finding adjacent Delone cells}
\label{Adjacent}

We want to determine the facets of a full-dimensional Delone cell,
which is given by its vertex set.
This representation conversion problem can be solved by many different
methods. For details and implementations we refer to 
{\tt cdd} \cite{cdd}, {\tt lrs} \cite{lrs},
{\tt pd} \cite{pd} and {\tt porta} \cite{porta}.

In order to exploit the symmetries we use the {\em adjacency decomposition
method} (see \cite{CR,bremner,perfectdim8}). It allows to compute
a complete list of inequivalent facet representatives: We compute an
initial facet by linear programming and insert it into the list 
of orbit representatives of facets.
From any such orbit, we compute the list of facets adjacent to a
representative and insert it, if necessary, into the list of
representatives until all orbits have been treated.
Computing adjacent facets is itself a representation conversion problem
in one dimension lower.
So this method can be applied recursively 
(see \cite{bremner,perfectdim8}). Note that our main algorithm is itself
an adjacency decomposition method.

After the computation of facets, we can compute adjacent full-dimensional
Delone cells: 
We take an initial vertex $v$
and so get a tentative empty sphere. If the sphere is not empty, then we
find another vertex $v$ and iterate until the sphere is indeed empty.
Figure \ref{IterativeAdjacent} illustrates this algorithm.

\begin{figure}
\begin{center}
\resizebox{8.7cm}{!}{\includegraphics[bb=0 0 1639 650,clip]{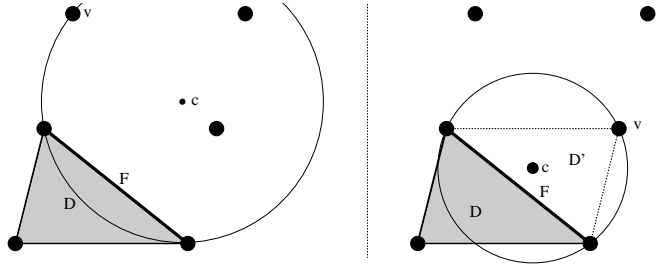}}\par
\end{center}
\caption{Finding $D'$, the full-dimensional Delone cell adjacent to $D$ at $F$}
\label{IterativeAdjacent}
\end{figure}

\begin{flushleft}
\smallskip
\textbf{Input:} $n$, $B \in \Q^{n \times n}$ matrix of rank $n$,
a full-dimensional Delone cell $D$ and a facet $F$ of $D$.\\
\textbf{Output:} vertex set ${\mathcal V}'$ of a full-dimensional Delone cell $D'$ with $D\cap D'=F$.\\
\smallskip
$\phi\leftarrow$ affine function on $\RR^n$ with $F=\{x\in D\mbox{~}:\mbox{~}\phi(x)=0\}$ and $\phi(x)> 0$\\
\hspace{4ex} on $D - F$.\\
${\mathcal V}_F\leftarrow$ vertices of $D$ belonging to $F$.\\
$v\leftarrow$ a point of $B\Z^n$ with $\phi(v)<0$.\\
\textbf{do}\\
\hspace{2ex} $S(c,r)\leftarrow$ sphere around ${\mathcal V}_F\cup \{v\}$.\\
\hspace{2ex} ${\mathcal V}'\leftarrow$ closest vectors in $B\Z^n$ to $c$.\\
\hspace{2ex} \textbf{if} $\Vert v'-c \Vert =r$ for a $v'\in {\mathcal V}'$ \textbf{then}\\
\hspace{2ex} \hspace{2ex} return ${\mathcal V}'$.\\
\hspace{2ex} \textbf{end if}\\
\hspace{2ex} $v\leftarrow$ one element of ${\mathcal V}'$.\\
\textbf{end do}\\
\end{flushleft}

One way to speed up the convergence of this algorithm in practice
is to heuristically choose an initial
vector $v$ with a sphere $S(c,r)$ of small radius.

\subsection{Checking equivalence}\label{Equivalence}

We have to test equivalence and compute stabilizers under the group $\Iso(L)$
of Delone cells of different dimensions. Below we propose three different
methods for this.

We can encode a Delone cell $D$ by the center $c(D)$ of the empty sphere
around it or by the {\em vertex barycenter} $g(D)=\frac{1}{|\vertt D|}\sum_{v\in \vertt D} v$ of its vertex set $\vertt D$.
Both $c(D)$ and $g(D)$ are invariant under the stabilizer of $D$.
Any two full-dimensional Delone cells $D$, $D'$ are equal
if and only if $c(D)=c(D')$.
However, it is possible if $n\geq 3$ that $c(D)$ lies outside
or on the boundary of $D$.
If $c(D)$ lies on the boundary of $D$ then a facet containing $c(D)$, 
which is itself a Delone cell, has the same center as $D$.
Hence, the sphere centers can be used to distinguish full-dimensional Delone 
cells but they do not distinguish Delone cells.
Therefore, we use the vertex barycenter.

In the first method we consider the classes of the vertex barycenters
$g(D)$ and $g(D')$ in the quotient $\QuotS{\RR^n}{L}$ and check
their equivalence
under the finite group $\QuotS{\Iso(L)}{L}\simeq \Orth(L)$.
The generic methods underlying isomorphism and stabilizer computations
generate the full orbit of $g(D)$ under $\QuotS{\Iso(L)}{L}$.
This is typically memory intensive. In some cases we can use a
method from computational group theory, which we now explain in an example.
Suppose $g(D)$ is expressed in a basis $(b_1,\ldots,b_n)$ of $L$
as $(\frac{\alpha_1}{2\cdot 3}, \ldots, \frac{\alpha_n}{2\cdot 3})$
with $0\leq \alpha_i\leq 5$ and we want to compute its stabilizer
under the group $\QuotS{\Iso(L)}{L}$.
The vector $2g(D)$ reduced modulo $L$ is expressed as
$(\frac{\tilde{\alpha}_1}{3}, \ldots, \frac{\tilde{\alpha}_n}{3})$
with $0\leq \tilde{\alpha}_i\leq 2$.
We first compute the stabilizer $H$ of the vector $2g(D)$ under
the action of $\QuotS{\Iso(L)}{L}$.
The stabilizer of $g(D)$ under $\QuotS{\Iso(L)}{L}$ is equal to
the stabilizer of $g(D)$ under $H$.
This method generalizes to more than two prime factors and it
is more memory efficient because
the generated orbits are smaller.

The second method uses finite metric spaces of the vertex set of
full-dimensional Delone cells obtained from the metric
$\Vert v-v'\Vert^2$ \cite[Chapter 14]{DL}.
A finite metric space defines an edge-weighted graph. Testing
if two edge weighted graphs are isomorphic can be reduced 
to testing if two vertex-weighted graphs are isomorphic
(see \cite[Page 25]{nauty-manual}).
In practice, the program {\tt nauty} \cite{nauty-manual} can solve the
isomorphism problem
if the number of vertices of $D$ and $D'$ is not too large.
If the metric spaces are not isomorphic, then $D$ and $D'$ are
not equivalent under $\Iso(L)$.
If they are isomorphic then every graph isomorphism corresponds to
a linear isometry between $D$ and $D'$ \cite{bremner,perfectdim8}.
For each of those isomorphisms, we check if it belongs to $\Iso(L)$.
This method is useful when the isometry group of $D$ is small.

For the third method, we use laminations over the
$n$-dimensional lattice $L$.
Let $D$ be a Delone cell of $L$ with vertex barycenter $g$.
One defines an $(n+1)$-dimensional lattice $L(g)$ by embedding
$L\subseteq \R^n$ into $\RR^{n+1}$ and adding layers to it
\begin{equation*}
L(g) = \left\{\alpha v + h : h \in L, \alpha \in \ZZ\right\} \subseteq \RR^{n+1},
\end{equation*}
where $v \in \RR^{n+1}$ is chosen so that $v+g$ is orthogonal to the
space spanned by $L$ and normalized so that $\|v + g\| = 1$.
A variant
of this construction is used for example to build the laminated
lattices, see \cite[Chapter 6]{conwaysloane}.
If $\phi$ is an element of $\Orth(L(g))$ preserving every layer of
the lamination, then
it maps the vector $v$ to some vector $w=v+h$ with $h\in L$.
The function $x\mapsto \phi(x)+h$ preserves the Delone cell and every
element preserving the layers is obtained in this way.
In practice, we can use the program {\tt AUTO} 
(see \cite{plesken}) of the package {\tt CARAT} (see \cite{carat-1998}) for
computing this automorphism group. The isomorphism problem is treated
similarly using the program {\tt ISOM}.

\subsection{Working with Gram matrices and periodic structures}\label{grammatrices}

In many cases, it is more convenient to work with the Gram matrix
$B^{T} B$ instead of the lattice basis $B$
(see \cite[Chapter 2.2]{conwaysloane}).
For instance, when $B$ is irrational but $B^T B$ is rational.
Note that our algorithms can be reformulated in terms of
Gram matrices.
Note also that all our algorithms can be modified to
deal with periodic point sets, that is for finite unions of
lattice translates. Our implementation is available from 
\cite{polyhedral}.

\subsection{Comparison}\label{sec:comparison}

In \cite{diamond} Viterbo and Biglier describe another algorithm for
computing the Voronoi cell of a lattice, called the diamond cutting
algorithm.
As in our approach they start with a parallelepiped $P$ defined by
the basis vectors.
Then they determine all lattice vectors which lie in a sphere containing
$2P$. This set contains all facet defining lattice vectors of $\DV(L)$.
Successively they add cutting planes obtained from these vectors
and update the complete face lattice of the tentative Voronoi cell.
They terminate when its volume coincides with $\det L$.
Their implementation uses floating point arithmetic.

In comparison, our approach has the following advantages:
We use the presence of symmetry in an efficient way.
We do not need to compute a huge initial list of potential facet defining
lattice vectors.
Our algorithm does not need to compute the face lattice,
not even for computing the quantizer constant as explained below.
Our implementation uses rational arithmetic only.

\section{Computing quantizer constants}
\label{sec:quantizer}

Recall from the introduction that the quantizer constant of a lattice $L$ is 
the integral
\begin{equation*}
G(L) = (\det L)^{-(1+2/n)} \int_{\DV(L)} \|x\|^2 dx.
\end{equation*}
A standard method for computing the integral $G(L)$ is to decompose $\DV(L)$ into simplices.
Suppose that $S$ is a simplex with vertices $v_1, \ldots, v_{n+1}$ in $\RR^n$.
Then (see \cite[Chapter 21, Theorem 2]{conwaysloane}) the following holds:
\begin{equation*}
\int_{S} \Vert x\Vert^2 dx=\frac{\vol S}{(n+1)(n+2)}\left(\left\Vert \sum_{i=1}^{n+1} v_i\right\Vert^2+\sum_{i=1}^{n+1} \Vert v_i\Vert^2\right).
\end{equation*}
Thus $G(L)$ can be obtained by summing the integrals of all simplices 
in a decomposition of $\DV(L)$.
Several practical methods for decomposing a polytope into simplices
are discussed in \cite{fukudaintegral}. In our implementation, 
we use the triangulation obtained by the program {\tt lrs}.
However, this method as well as the other methods explained in 
\cite{fukudaintegral} are sometimes impractical and 
they do not use symmetries.

In order to get a group invariant decomposition, we can use the barycentric
subdivision of $P$.
That is, given any flag $F_0\subset F_1\subset\ldots\subset F_n=P$ of
faces of $P$, we associate the simplex with vertex set
$g_0, g_1,\ldots, g_n$ where $g_i$ is the vertex barycenter of $F_i$.
Note that in general there is a difference between the barycenter 
$\frac{1}{\vol P}\int_P x dx$ of a polytope $P$ and its vertex barycenter
$\frac{1}{|\vertt P|}\sum_{v\in \vertt P} v$.
The group acts on the barycentric subdivision and the stabilizer of each
simplex is trivial. In practice, the number of orbits of flags
can be too large. 

We propose a hybrid approach, which combines the benefits of
both methods. Let ${\mathcal F}$ be the facet set of an $n$-dimensional
polytope $P$.
We can assume without loss of generality that $P$ has the origin as
its vertex barycenter. We then have
\begin{equation}\label{FirstSplitting}
\int_{P} \Vert x\Vert^2dx = \sum_{F\in {\mathcal F}} \int_{\conv(F,0)} \Vert x\Vert^2 dx.
\end{equation}
To compute this sum, it is sufficient to compute the integrals only for
orbit representatives of facets.
Let $F$ be a facet of $P$ and $p_F$ be a point in the affine space
spanned by $F$.
Then we can transform the integral over the cone $\conv(F, 0)$
in the following way:
\begin{equation*}
\begin{split}
\int_{\conv(F, 0)} \Vert x\Vert^2 dx=\frac{1}{n+2}\left(\int_{F} \Vert y-p_F\Vert^2 dy\right.\\
\left.+2 \int_{F} (y-p_F)\cdot p_F\, dy+\vol F \Vert p_F\Vert^2\right) .
\end{split}
\end{equation*}
If $p_F$ is the orthogonal projection of the origin $0$ onto $F$ then
the second summand vanishes. This point may not be invariant
under the automorphism group of the facet $F$, but the vertex barycenter
is.
If we use the vertex barycenter, we also have to compute 
the barycenter of the polytope $F$ as well as the volume
and the square integral.
In order to use symmetries coming from non-orthogonal linear transformations
of $P$, we use the matrix valued integral
\begin{equation*}
I_{0,1,2}(P)=\int_P \left(\begin{array}{c}
1\\
x
\end{array}\right) (1, x^t)\,dx.
\end{equation*}
This integral splits according to
\begin{equation*}
I_{0,1,2}(P)=\left(\begin{array}{cc}
I_0(P)   & I_1(P)^t\\
I_1(P) & I_2(P)
\end{array}\right),
\end{equation*}
where 
\begin{equation*}
I_{0}(P)=\int_P\,dx=\vol P, \quad
I_{1}(P)=\int_P x\, dx, \quad
I_{2}(P)=\int_P x x^t\, dx.
\end{equation*}
Let $G$ be a group of automorphisms of $P$.
If $g\in G$ acts on $\RR^n$ as $x\mapsto Ax+v$ then we define $H(g)=\left(\begin{array}{cc}
1  & 0\\
v  & A
\end{array}\right)$ the corresponding $(n+1)\times (n+1)$ matrix acting on homogeneous coordinates.
Let $O_1, \ldots, O_r$ be the $G$-orbits of facets of $P$,
with representatives $F_1, \ldots, F_r$.
Then the integral $I_{0,1,2}(P)$ simplifies to
\begin{equation*}
I_{0,1,2}(P)=\sum_{i=1}^{r} |O_i|\left(\frac{1}{|G|}\sum_{g\in G} H(g) I_{0,1,2}(\conv(F_i, 0)) H(g)^t\right).
\end{equation*}

Assume that $I_{0,1,2}(\conv(F_i, 0))$ is already computed. 
To compute the sum in the parenthesis, we first incrementally compute a basis
of the affine hull of the orbit $\{H(g) I_{0,1,2}(\conv(F_i, 0)) H(g)^t : g\in G\}$.
The only $G$-invariant element of the affine hull is the sum we
want to compute.

We now want to compute $I_{0,1,2}(\conv(F, 0))$ in terms of lower
dimensional integrals.
The integral depends on the chosen basis. If $f$ is an affine transformation
of $\RR^n$, then the change of variables formula for integrals gives
\begin{equation*}
H(f) I_{0,1,2}(fP) H(f)^t |\det H(f)|=I_{0,1,2}(P),
\end{equation*}
for any an $n$-dimensional polytope $P$ in $\RR^n$.
This allows to compute $I_{0,1,2}(P)$ for another basis.
So, we can choose a coordinate system such that
\begin{equation*}
F=\left\lbrace\left(\begin{array}{c}
1\\
x
\end{array}\right) :  x\in F'\right\rbrace\subset \RR^n,
\end{equation*}
where $F'\subset \RR^{n-1}$ is an $(n-1)$-dimensional polytope.
We then have the following formulas:
\begin{equation*}
\begin{array}{c}
I_0(\conv(F,0))=\frac{1}{n}I_0(F'),\quad
I_1(\conv(F,0))=\frac{1}{n+1}\left(\begin{array}{c}
I_0(F')\\
I_1(F')
\end{array}\right),\\
\quad
I_{2}(\conv(F, 0))=\frac{1}{n+2}I_{0,1,2}(F').
\end{array}
\end{equation*}
For computing $I_{0,1,2}(F')$, we have two options: Either we use the
first method of this section, which involves computing a triangulation
or we apply the above method recursively.
The decision is made heuristically, depending on the size of the
automorphism group of $F$ and its number of vertices.
In order to reduce the size of the computation, one can store
intermediate results.

Those methods are general and apply to any polytope and any polynomial
function, which we want to integrate over $P$.
Note that a similar method of using the standard formula
(\ref{FirstSplitting}) has been used
for computing the volume
in \cite{fukudaintegral} under the name of {\em Lasserre's method}
(\cite{lasserre}), albeit in a non-group setting.

\section{Results}\label{sec:results}

In this section, we collect results from our implementation
of the algorithms explained in Sections \ref{sec:algorithm}
and \ref{sec:quantizer}.
We obtain previously unknown exact covering densities and 
quantizing constants of several prominent lattices and their duals.
Recall that the \textit{dual} $L^*$ of a lattice $L\subset \RR^n$ is defined by
\begin{equation*}
L^*=\left\{x\in\RR^n : y\cdot x\in \ZZ \quad \mbox{for~all~} y\in L\right\}.
\end{equation*}
The {\em covering density} of an $n$-dimensional lattice $L$ is
\begin{equation*}
\frac{\mu(L)^n}{\det L} \vol B_n,
\end{equation*}
where $B_n$ is the unit ball in $\RR^n$.
Other computations of Voronoi cells of lattices can be found in \cite{cellstructure}, 
\cite[Chapter 5]{engelbook} and \cite{patera}.
All computations are done in exact rational arithmetic.
In the tables the covering densities are given in floating
point; the exact expressions would be too large.

\subsection{Coxeter lattices}

The root lattice $\mathsf{A}_n$ is defined by
\begin{equation*}
\mathsf{A}_n=\left\{x\in \ZZ^{n+1}\mbox{~}:\mbox{~}\sum_{i=1}^{n+1}x_i=0\right\}.
\end{equation*}
If $r$ divides $n+1$, the \textit{Coxeter lattice} $\mathsf{A}_n^r$
(see \cite{coxeter-1951}) is defined by translates of $\mathsf{A}_n$:
\begin{equation*}
\mathsf{A}_n^r=\mathsf{A}_n\cup (v_{n}^r+\mathsf{A}_n) \cup \ldots\cup ((r-1)v_{n}^r+\mathsf{A}_n),
\end{equation*}
where $v_n^r=\frac{1}{n+1}\sum_{i=2}^{n+1}(e_i-e_1)$.
The dual lattice of $\mathsf{A}_n^r$ is $\mathsf{A}_n^{n+1/r}$.

The Delone decomposition of the lattice $\mathsf{A}_n^r$ has been studied
in \cite{anzin-2002,anzin-2006,baranovskii-1994} up to dimension $n=15$,
hoping to obtain lattices with low covering density.
One pleasant fact is that the symmetry group of $\mathsf{A}_{n}^r$
contains the group $\Sym(n+1)\times \ZZ_2$.
Latter can be represented as a permutation group acting on $n+3$ points,
which drastically simplifies isomorphism computations.

In Table \ref{TableCoxeterLattices} we list the obtained results.
Note that the lattices $\mathsf{A}_{17}^6$, $\mathsf{A}_{19}^{10}$, 
$\mathsf{A}_{20}^{7}$ and $\mathsf{A}_{21}^{11}$ turn out to give
new record sphere coverings. 
Up to dimension $8$ all those lattices are well known and their Voronoi
cells can be obtained by standard computer algebra software.
Our list is complete up to dimension $21$. For the missing cases
we could not finish the computation.

\begin{table}
\begin{center}
\begin{tabular}{c|c|c|c|c|c}
lattice    & \# orbits & covering density
    & lattice    & \# orbits & covering density\\[0.2mm]\hline
&&&&&\\[-3mm]
$\mathsf{A}_{9}^2$     &6     &$18.543333$& $\mathsf{A}_{9}^5$     &5     &$4.340184$\\[0.4mm]\hline
&&&&&\\[-3mm]
$\mathsf{A}_{11}^{2}$  &6     &$94.090996$& $\mathsf{A}_{11}^{3}$  &11    &$27.089662$\\[0.7mm]
$\mathsf{A}_{11}^{4}$  &16    &$5.598337$& $\mathsf{A}_{11}^{6}$  &4     &$7.618558$\\[0.4mm]\hline
&&&&&\\[-3mm]
$\mathsf{A}_{13}^{2}$  &10    &$134.623484$& $\mathsf{A}_{13}^{7}$  &10    &$7.864060$\\[0.4mm]\hline
&&&&&\\[-3mm]
$\mathsf{A}_{14}^{3}$  &17    &$32.313517$& $\mathsf{A}_{14}^{5}$  &31    &$9.006610$\\[0.4mm]\hline
&&&&&\\[-3mm]
$\mathsf{A}_{15}^{2}$  &10    &$722.452642$& $\mathsf{A}_{15}^{4}$  &19    &$25.363859$\\[0.7mm]
$\mathsf{A}_{15}^{8}$  &10    &$11.601626$& &&\\[0.4mm]\hline
&&&&&\\[-3mm]
$\mathsf{A}_{17}^{2}$  &15    &$1068.513081$ &$\mathsf{A}_{17}^{3}$  &26    &$240.511580$\\[0.7mm]
$\mathsf{A}_{17}^{6}$  &73    &$12.357468$ & $\mathsf{A}_{17}^{9}$  &24    &$17.231927$\\[0.4mm]\hline
&&&&&\\[-3mm]
$\mathsf{A}_{19}^{2}$  &15    &$5921.056764$ & $\mathsf{A}_{19}^{4}$  &58    &$40.445924$\\[0.7mm]
$\mathsf{A}_{19}^{5}$  &80    &$25.609662$ & $\mathsf{A}_{19}^{10}$ &80    &$21.229200$\\[0.4mm]\hline
&&&&&\\[-3mm]
$\mathsf{A}_{20}^{3}$  &40    &$307.209487$ & $\mathsf{A}_{20}^{7}$  &187   &$20.366828$\\[0.4mm]\hline
&&&&&\\[-3mm]
$\mathsf{A}_{21}^{2}$  &21    &$8937.567486$ & $\mathsf{A}_{21}^{11}$ &64    &$27.773140$\\[0.4mm]\hline
&&&&&\\[-3mm]
$\mathsf{A}_{23}^{3}$  &55    &$2405.032746$ & $\mathsf{A}_{23}^{4}$  &85    &$205.561225$\\[0.7mm]
$\mathsf{A}_{23}^{6}$  &187   &$79.575330$ & $\mathsf{A}_{23}^8$  &495    &$31.858162$\\[0.7mm]
$\mathsf{A}_{23}^{12}$ &100   &$43.231587$&&\\[0.4mm]\hline
&&&&&\\[-3mm]
$\mathsf{A}_{24}^{5}$  &144  &$115.011591$&&\\[0.4mm]\hline
&&&&&\\[-3mm]
$\mathsf{A}_{25}^{13}$ &210   &$54.472182$&&\\[0.4mm]\hline
&&&&&\\[-3mm]
$\mathsf{A}_{26}^{3}$  &75  &$3184.1387034$&$\mathsf{A}_{26}^{9}$  &1231  &$50.937168$\\[0.4mm]\hline
&&&&&\\[-3mm]
$\mathsf{A}_{27}^{4}$  &156  &$350.137031$  &$\mathsf{A}_{27}^7$ &650& $81.869181$\\[0.4mm]\hline
&&&&&\\[-3mm]
$\mathsf{A}_{27}^{14}$ &338   &$76.909712$&&\\[0.4mm]\hline
&&&&&\\[-3mm]
$\mathsf{A}_{29}^{3}$  &102   &$25664.644103$& $\mathsf{A}_{29}^{5}$  &347   &$202.040331$\\[0.7mm]
$\mathsf{A}_{29}^{6}$  &711   &$154.329831$&$\mathsf{A}_{29}^{10}$  &3581  &$84.324725$\\[0.7mm]
$\mathsf{A}_{29}^{15}$ &678   &$114.084219$&&\\[0.4mm]\hline
&&&&&\\[-3mm]
$\mathsf{A}_{31}^{16}$ &1225  &$33.934941$&                &      &\\
\end{tabular}
\end{center}
\vspace{2.2mm}
\caption{Number of orbits of full-dimensional Delone cells and covering density for some Coxeter lattices}
\label{TableCoxeterLattices}
\end{table}

\subsection{Laminated lattices}
The laminated lattices, which are defined in \cite[Chapter 6]{conwaysloane},
give the best known lattice sphere packings in many dimensions.
The Delone subdivision is known up to dimension $8$ and in dimension $24$ 
for all laminated lattices and their duals \cite[Chapters 21, 23, 25]{conwaysloane}.
In dimension $16$, the covering density of $\Lambda_{16}$ is known
\cite[Chapter 6]{conwaysloane}.

In Table \ref{LaminatedDelaunays} we list the
obtained results, which are complete up to dimension $17$.

\begin{table}
\begin{center}
\begin{tabular}{c|c|c|c|c|c}
lattice   & \# orbits  & covering density & lattice & \# orbits & covering density\\
\hline
$\Lambda_9$          &5    &$9.003527$              &$\Lambda_9^{*}$ &9 &$9.003527$\\[0.4mm]\hline
$\Lambda_{10}$       &7    &$12.408839$              &$\Lambda_{10}^{*}$ &21 &$9.306629$\\[0.4mm]\hline
&&&&&\\[-3mm]
$\Lambda_{11}^{max}$ &11   &$24.781167$              &$\Lambda_{11}^{max*}$ &18 &$19.243468$\\[0.7mm]
$\Lambda_{11}^{min}$ &18   &$24.781167$              &$\Lambda_{11}^{min*}$ &153 &$8.170432$\\[0.4mm]\hline
&&&&&\\[-3mm]
$\Lambda_{12}^{max}$ &5    &$30.418954$              &$\Lambda_{12}^{max*}$ &8 &$42.728408$\\[0.7mm]
$\Lambda_{12}^{mid}$ &23   &$30.418954$              &$\Lambda_{12}^{mid*}$ &52 &$19.176309$\\[0.7mm]
$\Lambda_{12}^{min}$ &13   &$30.418954$              &$\Lambda_{12}^{min*}$ &78 &$12.292973$\\[0.4mm]\hline
&&&&&\\[-3mm]
$\Lambda_{13}^{max}$ &18   &$60.455139$              &$\Lambda_{13}^{max*}$ &57 &$43.214494$\\[0.7mm]
$\Lambda_{13}^{mid}$ &46   &$35.931846$              &$\Lambda_{13}^{mid*}$ &125 &$19.155991$\\[0.7mm]
$\Lambda_{13}^{min}$ &129  &$60.455139$              &$\Lambda_{13}^{min*}$ &5683 &$13.724864$\\[0.4mm]\hline
&&&&&\\[-3.5mm]
$\Lambda_{14}$       &65   &$98.875610$              &$\Lambda_{14}^{*}$ &1392 &$34.721750$\\[0.4mm]\hline
&&&&&\\[-3.5mm]
$\Lambda_{15}$       &27   &$202.910873$              &$\Lambda_{15}^{*}$ &108 &$25.642067$\\[0.4mm]\hline
&&&&&\\[-3.5mm]
$\Lambda_{16}$       &4    &$96.500266$              &$\Lambda_{16}^{*}$ &4 &$96.500266$\\[0.4mm]\hline
&&&&&\\[-3.5mm]
$\Lambda_{17}$       &28   &$197.719499$              &$\Lambda_{17}^{*}$ &720 &$100.173101$\\[0.4mm]\hline
&&&&&\\[-3.5mm]
$\Lambda_{18}$       &239  &$301.192334$              &   &   &\\[0.4mm]\hline
&&&&&\\[-3.5mm]
$\Lambda_{23}$       &709  &$7609.03133$              &   &   &\\
\end{tabular}
\end{center}
\vspace{2.2mm}
\caption{Number of orbits of full-dimensional Delone cells and covering density for some laminated lattices and their duals}
\label{LaminatedDelaunays}
\end{table}

\subsection{Shorter Leech lattice}

The $4600$ shortest vectors of $\Lambda_{23}^*$ of
define a sublattice of index $2$, called the
{\em shorter Leech lattice} $O_{23}$ (\cite[Page 179, 420, 441]{conwaysloane}).
The Delone decomposition (see Table \ref{O23delaunays}) is remarkable in many
respects: There are only $5$ orbits and the first one has the full symmetry
group of the lattice.
It turns out that $\Lambda_{23}^*=O_{23}\cup (v+O_{23})$ where $v$ is the
center of a centrally symmetric Delone cell lying in the first orbit.
The covering density of $O_{23}$ is $15218.062669$.

\begin{table}
\begin{center}
\begin{tabular}{c|c}
number of vertices    & size of stabilizer group\\
\hline
$94208$  &   $84610842624000$\\
$32$     &   $1344$\\
$24$     &   $10200960$\\
$24$     &   $1320$\\
$24$     &   $1320$
\end{tabular}
\end{center}
\caption{Orbits of full-dimensional Delone cells of $O_{23}$}
\label{O23delaunays}
\end{table}

\subsection{Cut lattices}
The cut polytope $\CUT_n$ is a famous polytope appearing in combinatorial
optimization (see \cite{DL}).
It has $2^{n-1}$ vertices and is of dimension $\frac{n(n-1)}{2}$.
The lattice generated by its vertices is called
{\em cut lattice} and is denoted by $L(\CUT_n)$ (see \cite{cutlattice}).
The polytope $\CUT_n$ is one of its full-dimensional Delone cells.
We list our results in Table \ref{CutLatticesTable}.

\begin{table}

\begin{center}
\begin{tabular}{c|c|c|c}
lattice   & dimension & \# orbits & covering density\\
\hline
$L(\CUT_3)$ &3        &$2$   &$2.09439$\\
$L(\CUT_4)$ &6        &$4$   &$5.16771$\\
$L(\CUT_5)$ &10       &$12$  &$40.80262$\\
$L(\CUT_6)$ &15       &$112$ &$255.4255$\\
\end{tabular}
\end{center}

\caption{Dimensions, number of orbits of full-dimensional Delone cells and covering density of some cut lattices}
\label{CutLatticesTable}
\end{table}

\subsection{Quantizer constants}

In Table \ref{QuantConstantTable} we collect some new exact
quantizer constants.

According to \cite{eriksson}, the lattice $\mathsf{D}_{10}^+$ is conjectured
to be the optimal lattice quantizer.
Conway and Sloane approximated $G(\mathsf{K}_{12})$ (\cite[Table 2.3]{conwaysloane}) using Monte-Carlo integration; 
our exact computation fits into their bounds.

\begin{table}
\begin{center}
\begin{tabular}{c|c}
lattice           &  quantizer constant\\[1mm]\hline
&\\[-3mm]
$\Lambda_9$       & $\frac{151301}{2099520}\approx 0.07206$\\[2mm]
$\Lambda_9^*$     & $\frac{1371514291}{19110297600}\approx 0.07176$\\[2mm]
$\mathsf{A}_9^2$  & $\frac{2120743}{\sqrt[9]{5.2^8}13271040}\approx 0.072166$\\[2.7mm]
$\mathsf{A}_9^5$  & $\frac{8651427563}{\sqrt[9]{2.5^8}26578125000}\approx 0.072079$\\[2.7mm]
$\mathsf{D}_{10}^{+}$      & $\frac{4568341}{64512000}\approx 0.07081$\\[2mm]
$\mathsf{A}_{11}^2$ &$\frac{452059}{\sqrt[11]{3}5702400}\approx 0.07174$\\[2.7mm]
$\mathsf{A}_{11}^3$ &$\frac{287544281699}{\sqrt[11]{4.3^{10}}1325839006800}\approx 0.070426$\\[2.7mm]
$\mathsf{A}_{11}^4$ &$\frac{6387657954959}{\sqrt[11]{3.2^{9}}46506442752000}\approx 0.070494$\\[2.7mm]
$\mathsf{D}_{12}^+$ &$\frac{29183629}{412776000}\approx 0.070700$\\[2mm]
$\mathsf{K}_{12}$  &$\frac{797361941}{\sqrt{3}6567561000}\approx 0.070095$
\end{tabular}
\end{center}

\caption{Quantizer constants of some lattices}
\label{QuantConstantTable}
\end{table}

\section*{Acknowledgements}

We thank David Avis,
Andreas Enge,
Alexander Hulpke, 
Gabriele Nebe,
Warren Smith,
Bernd Souvignier
and the anonymous referees for helpful discussions and suggestions.

\end{document}